\documentclass[11pt, reqno]{amsart}

\usepackage{amssymb,latexsym,amsmath,amsfonts}
\usepackage{mathrsfs}
\usepackage{graphicx}
\usepackage[usenames]{color}

\numberwithin{equation}{section}

\theoremstyle{definition}
\newtheorem{definition}{Definition}[section]
\newtheorem{fact}[definition]{Fact}

\theoremstyle{remark}
\newtheorem{remark}[definition]{Remark}

 \theoremstyle{plain}
\newtheorem{theorem}[definition]{Theorem}
\newtheorem{result}[definition]{Result}
\newtheorem{lemma}[definition]{Lemma}
\newtheorem{proposition}[definition]{Proposition}

\newtheorem{corollary}[definition]{Corollary}

\newcommand{\zt}{\zeta}

\newcommand{\que}{\mathcal{Q}}

\newcommand{\OM}{\Omega}
\newcommand{\D}{\mathbb{D}}

\newcommand{\hol}{\mathcal{O}}

\newcommand{\hcf}{{\sf gcd}}

\newcommand{\Cn}{\mathbb{C}^n}
\newcommand{\C}{\mathbb{C}} 

\newcommand{\Nn}{\mathbb{N}^n}
\newcommand{\Tn}{\mathbb{T}^n}
\makeatletter
\newcommand*{\rom}[1]{\expandafter\@slowromancap\romannumeral #1@}
\makeatother

\begin{document}

\title[The three-point Pick interpolation problem]{The three-point Pick--Nevanlinna interpolation problem on the polydisc}

\author{Vikramjeet Singh Chandel}
\address{Department of Mathematics, Indian Institute of Science, Bangalore 560012, India}
\email{abelvikram@math.iisc.ernet.in}

\thanks{This work is supported by a UGC Centre for
Advanced Study grant and by a scholarship from the IISc}


\begin{abstract}
We give a characterization for the existence of a holomorphic interpolant on 
the unit polydisc $\D^n,$ $n\geq 2,$ for prescribed three-point Pick--Nevanlinna data.
One of the key steps is a characterization for the existence of an interpolant that is
a rational inner function on $\D^n.$ The latter reduces
the search for a three-point interpolant to finding a single rational
inner function that satisfies a type of positivity condition and arises
from a polynomial of a very special form.
This in turn relies on a pair of results, which are of independent interest,
on the factorization of rational inner functions.
\end{abstract}
\maketitle

\section{Introduction and statement of results}\label{S:intro}
The problem alluded to in the title of this work is the following
(in this work, $\D$ will denote the open unit disc with centre $0\in\C$):
\begin{itemize}
\item[$(*)$] Let $X_1,\dots,X_{N}$ be distinct points in $\D^n$ and let $w_1,\dots,w_{N}\in\D.$
Characterize those data $\{(X_j,w_j):1\leq j\leq N\}$ for which there exists a holomorphic function
$\Phi:\D^n\longrightarrow\D$ such that $\Phi(X_j)=w_j, \ j=1,\dots,N.$
\end{itemize}
\noindent This, in the case $n=1$, was solved by Pick in $1916$ and the properties 
of an interpolant $\Phi,$ whenever it exists, were studied by Nevanlinna. Sarason's proof 
\cite{sarason:ftp67} opened up a new paradigm 
for approaching $(*)$ for $n\geq 2.$ This approach led to Agler's solution to {\em a version of $(*)$}:
characterizing those $\{(X_j,w_j):1\leq j\leq N\},$ for any $n\geq 2,$ that admit
an interpolant in the Schur--Agler class: see \cite[Theorem 11.90]{agmac:ftp02}.
This stems from Agler's solution \cite{ag:ftp88} of $(*)$ for $n=2.$ Ball and Trent in \cite{ball:ftp98}
provided a simpler proof of Agler's solution\,---\,also see \cite{ball:ftp99}\,---\,and found a
parametrization of all Schur--Agler-class interpolants.
\smallskip

Agler's solution to $(*)$ for $n=2$ relies on And\^{o}'s inequality \cite{ando:ftp63}
(see also the article \cite{agmac:ftp99} by Agler--McCarthy). For $n\geq 3,$ the
Schur--Agler class is {\em strictly} smaller than the class $\{\Phi\in H^{\infty}(\D^n):\sup_{\D^n}|\Phi|\leq 1\}$
(i.e., the Schur class). There have thus been several articles in the last two decades
that have dwelt on the problem $(*).$ Apart from the aforementioned works, we refer the reader
to the articles \cite{cw:ftp94}, \cite{ham:ftp013} and to the works listed in the references therein.
However, despite all the results obtained so far:
\begin{itemize}
\item[$(i)$] The known characterizations for a Schur-class interpolant, when $n\geq 2,$ involve 
searching for a positivity condition through a {\em large} space of parameters.
\item[$(ii)$] Until very recently, one had little knowledge of the structure of
the interpolant $\Phi\in\hol(\D^n;\D)$ whenever it exists.
\end{itemize}
\smallskip

A class of functions in which one may look for an interpolant for the data $\{(X_j,w_j):1\leq j\leq N\}$
(or, alternatively, conclude that there is no such interpolant in that class) is the class of 
rational inner functions on $\D^n.$ This would certainly address the concern $(ii)$ above: there is a lot
that one knows about the structure of rational inner functions. We shall recall some of these properties at
the beginning of Section~\ref{S:prelim-res}. This viewpoint is strongly supported by developments
not long after the first version of this paper was written.

\begin{remark}\label{Rm:b}
Under the constraint $N=3,$  based on the work of Kosi{\'n}ski \cite{kos:ftp015}, 
Knese recently proved \cite[Section 4]{Knese:ftp015} that if a
Schur-class interpolant for $(*)$ exists, then an interpolant in the Schur--Agler class can be found. 
Schur--Agler-class functions in the polydisc have a well known realization formula: see \cite{Agler:ftp90} by Agler.
Bringing this realization formula to the argument underlying \cite[Corollary 2.13]{agmac:ftp99} one can prove (see the
remark on p.~193 of \cite{agmac:ftp99}) that a rational inner interpolant exists.
Our main theorem, which relies on the last two facts, is an outgrowth of an earlier version\,---\,of June
$2015$\,---\,of this paper \cite{chand:ftp015}.
\end{remark}

Further comment on our result, in view of Knese's work,
will be more meaningful once we have stated our main theorem.
To do so, we shall need some notations and terminology. Given a polynomial
$Q\in \C[z_1,\dots, z_n]$, recall that the {\em support of $Q$} is the set
\[
 {\sf supp}(Q) = \big\{\alpha\in \Nn : \ \text{\small $\dfrac{\partial Q}{\partial z^\alpha}(0)$}\neq 0\big\}.
\]
Writing $Q(z) =\sum_{j=0}^{d}\sum_{|\alpha|=j} a_{\alpha}z^{\alpha}$, define
(we use standard multi-index notation here)
\begin{align*}
 \widetilde{Q}(z)&:=\sum_{j=0}^{d}\sum_{|\alpha|=j}\overline{a_{\alpha}}z^{\alpha},\\
 \widetilde{Q}\left(\frac{1}{z}\right)&:=\sum_{j=0}^{d}\sum_{|\alpha|=j}\overline{a_{\alpha}}
 \frac{1}{z^{\alpha}},\\
 \nu(Q)&:=(\nu_1(Q),\dots,\nu_n(Q)),
\end{align*}
where $\nu_j(Q)$ denotes the
degree of the polynomial $Q(b_1,\dots,b_{j-1},\zeta,b_{j+1},\dots,b_n)\in\C[\zeta]$ for a {\em generic}
$(b_1,\dots,b_{j-1},b_{j+1},\dots,b_n)\in\C^{n-1}.$ We say that the polynomial $Q$ is {\em deficient
in degree} if the multi-index $\nu(Q)\notin {\sf supp}(Q)$ (our terminology stems from the fact that the
latter property is equivalent to $|\nu(Q)| > d$). We are now in a position to state our main theorem.
One final note: given $a\in\D,$ $\psi_a$ will denote the automorphism
\begin{equation}
\psi_a(z)=\frac{z-a}{1-\bar{a}z}, \ \ \ z\in\D,\label{E:aaut}
\end{equation}
and in what follows, given $X_3\in\D^n$ we define $\Psi_{X_3}\in\text{Aut}(\D^n)$ as
$\Psi_{X_3}\equiv(\psi_{X_{3,1}},\dots,\psi_{X_{3,n}}),$
where we write $X_3:=(X_{3,1},\dots,X_{3,n}).$

\begin{theorem}\label{T:mainThm}
Let $X_1,X_2,X_3$ be three distinct points in $\D^n,\ n\geq 2,$ and let $w_1,w_2,w_3\in\D.$
Then the following are equivalent:
\begin{itemize}
\item[\rom{1})] There exists $\Phi\in H^{\infty}(\D^n)$ such that $\sup_{\D^n}|\Phi|\leq 1$ and
$\Phi(X_j)=w_j,\ j=1,2,3.$
\item[\rom{2})] There exists a rational inner function $F$ on $\D^n$ such that $F(X_j)=w_j,\ j=1,2,3.$ 
\item[\rom{3})] There exists a rational inner function $H$ on $\D^n$ such that
\[
 w_j'/H(X_j')\in\overline{\D}\;\;\;\text{for $j=1,2$},
\]
and is of either one of the following forms:
\[
 H(z) = \begin{cases}
 		z_j \ \ \text{for some $j:\,1\leq j\leq n$}, &\!\!\!\text{OR} \\
 		z^{\nu(Q)}\widetilde{Q}(\frac{1}{z})/Q(z), &{ \ }
 		\end{cases}
\]
where $Q$ is an irreducible polynomial having no zeros in $\D^n$ and is deficient in degree,
and there exists an integer $l\in\{1,2,\dots,n\}$ such that the $2\times 2$ matrix
\begin{equation}
{\begin{bmatrix}
\dfrac{1-({w_j'}/{H(X_j')})(\overline{{w_k'}/{H(X_k')}})}{1-X_{j,\,l}'\overline{X'}_{k,\,l}}
\end{bmatrix}}_{j,k=1}^{2}\label{E:mainmat}
\end{equation}
 is positive semi-definite. Here $w_j':=\psi_{w_3}(w_j),$ $X_j':=\Psi_{X_3}(X_j), \ j=1,2,$ and we write
 $X_j=(X_{j,\,1},\dots,X_{j,\,n}).$ 
\end{itemize}
Furthermore, if the last condition holds true, then:
\begin{itemize}
\item[$a)$] If the matrix in \eqref{E:mainmat} is zero, then $\exists c\in\partial{\D}$ such that
 $F={\psi}^{-1}_{w_3}\circ(cH)\circ\Psi_{X_3}$ is an interpolant for the above data.
\item[$b)$] If the rank of the matrix in \eqref{E:mainmat} is $r, \ r=1,2,$ then there is a
 Blaschke product $B$ of degree $r$ such that $F={\psi}^{-1}_{w_3}\circ((B\circ \pi_l)H)\circ\Psi_{X_3}$ is an
 interpolant for the above data
(here, $\pi_l$ denotes the projection onto the $l$-th coordinate, $l$ as above).
\end{itemize}
\end{theorem}
 
In view of Remark \ref{Rm:b}, it follows\,---\,from \cite[Theorem 11.90]{agmac:ftp02}, for
instance\,---\,that a necessary and  sufficient condition for the existence
of a Schur-class interpolant for the data in Theorem \ref{T:mainThm} is the existence of $n$ positive
semi-definite functions $\Gamma_l:\{X_1,X_2,X_3\}\times\{X_1,X_2,X_3\}\longrightarrow\C,\,\,l=1,\dots,n,$
such that
\begin{equation}\label{E:sa}
1-w_j\bar{w}_{k}=\sum_{l=1}^{n}\left(1-X_{j,\,l}\overline{X}_{k,\,l}\right)
\Gamma_{l}\left(X_j,X_k\right),\,\,j,k=1,2,3.
\end{equation}
So, as soon as $n\geq 2,$ the problem of determining the existence of an interpolant reduces to a 
quest for certain (unknown) positive semi-definite functions that satisfy the $6$ conditions in \eqref{E:sa}.
This is the issue alluded to in $(i)$ above. The purpose of presenting the equivalence of
(\rom{1}) and (\rom{3}) in Theorem \ref{T:mainThm} is to contribute to dealing with the issue $(i).$
By Theorem \ref{T:mainThm}, the problem of determining the existence of an interpolant reduces
to a {\em different kind of quest} for positivity, one in which:
\begin{itemize}
 \item some of the parameters of the matrix in \eqref{E:mainmat} are explicitly known;
 \item the unknown parameters range over a class of polynomials of a very special structure.
\end{itemize}

\begin{remark}\label{Rm:m}
The proof of the equivalence of (\rom{2}) and (\rom{3}) has some bearing on the issue stated in $(ii)$ above.
Namely: if, by some method, one has found a rational inner interpolant $F_1$ of very high degree,
the above proof suggests a method to look for a simpler rational inner interpolant $F_2$ for the given data. We refer
the reader to subsection~\ref{SS:impli} as well as to a related discussion therein on how/when the above considerations
arise.
\end{remark}

The alert reader will surmise that the idea of the proof of the implication (\rom{2})\,$\Rightarrow$\,(\rom{3}) is a form of
the Schur algorithm. (Indeed, there are no univariate polynomials that are deficient in degree, owing to
which the matrix in \eqref{E:mainmat} will, for $n = 1,$ be a matrix that the reader will recognize.)
The details behind this observation require some work. Indeed, this article is as much a study of certain properties
of rational inner functions on $\D^n$ as it is about Theorem~\ref{T:mainThm}.
The former is the content of Section~\ref{S:prelim-res}.
The proof of Theorem~\ref{T:mainThm} is given in Section~\ref{S:Thm}.

The reader will also discern that the proof of (\rom{2})\,$\Leftrightarrow$\,(\rom{3}) hints
at an extension\,---\,provided the search for interpolants is in the class of rational inner functions\,---\,for
$N$ points, $N\geq4.$ (We will be brief about this; see Remark \ref{Rm:E} for details).
\medskip

\subsection*{Acknowledgments}
The author wishes to thank Gautam Bharali for the many useful discussions
during the course of this work. He also thanks the anonymous referee of an earlier version of this work for the
many suggestions for improving the exposition\,---\,especially of Section~\ref{S:prelim-res}.
\medskip

\section{Some results about rational inner functions on $\D^n$}\label{S:prelim-res}

In this section, we shall present a couple of results about rational inner functions 
on the polydisc $\mathbb{D}^n.$ We shall make use of the notation introduced prior to 
Theorem~\ref{T:mainThm}. These notations help us present the following 
discussion about rational inner functions on $\D^n.$

\begin{fact}\label{SS:rif}
An {\em inner function} on $\D^n$ is a function $f\in H^{\infty}(\D^n)$ such that
$\lim_{r\to 1^{-}}|f(rw)|=1$ for almost every $w\in\mathbb{T}^n.$ A {\em rational inner function} on $\D^n$
is an inner function that is rational. It is elementary to see that, given a polynomial
$Q\in\C[z_1,\dots,z_n],$ any function of the form
\[ f(z)=\frac{Az^{\beta}\widetilde{Q}(\frac{1}{z})}{Q(z)},\]
where 
\begin{itemize}
\item $Z(Q)\cap\D^n=\emptyset,$
\item $z^{\beta}\widetilde{Q}(\frac{1}{z})$ is a polynomial,
\item $A$ is a unimodular constant,
\end{itemize}
is a rational inner function. Here, and in what follows, $Z(Q)$ denotes the zero set of $Q.$
Moreover, it is a fact \cite[Theorem 5.2.5]{rudin:ftp69} that every rational inner function on $\D^n$
has the above form.
\end{fact}

The next two results are central to proving Theorem~\ref{T:mainThm}, and also of independent interest.

\begin{proposition}\label{P:factprop}
Let $f$ be a nonconstant rational inner function of the form
${z^{\nu(Q)}\widetilde{Q}(\frac{1}{z})}/{Q(z)},$
where $Q$ is a nonconstant polynomial in $\Cn$ such that $Z(Q)\cap\D^n=\emptyset.$ Then:
\begin{itemize}
\item[$(a)$] There exist a nonconstant polynomial $\que$ with $Z(\que)\cap\D^n=\emptyset$
and a unimodular constant $C$ such that $f$ can also be expressed as
\begin{equation}
 f(z) = C\,\frac{z^{\nu(\que)}\widetilde{\que}(\frac{1}{z})}{\que(z)}, \label{E:altform}
\end{equation}
and such that the numerator and the denominator of the above expression have no (nonconstant) irreducible
polynomial factors in common.

\item[$(b)$] There exist rational
inner functions $f_1,f_2\in\hol(\D^n),$ both nonunits in $\hol(\D^n),$
such that $f=f_1f_2$ in $\D^n$ if and only if $\que$ is reducible in $\C[z_1,\dots,z_n].$ 
\end{itemize}
\end{proposition}

We call a nonconstant rational inner function $f$ on $\D^n$ an {\em irreducible inner function} (resp., {\em reducible})
if we cannot (resp., can) express it as $f=gh,$ where $g$ and $h$ are rational inner functions and nonunits
in $\hol(\D^n).$
We now have the following corollary to Proposition~\ref{P:factprop}.

\begin{corollary}\label{C:irdefpol}
Let $f$ be an irreducible rational inner function such that $f(0)=0.$ Then, either
 $f(z) = z_j$ for some $j\in \{1,\dots, n\}$, or it has the form (modulo scaling by a unimodular constant)
\[
 f(z) = z^{\nu(\que)}\widetilde{\que}\left(\frac{1}{z}\right)\big/\que(z),
\]
where $\que$ is an irreducible polynomial having no zeros in $\D^n$ and is deficient in degree.
\end{corollary}

The corollary is immediate from Proposition \ref{P:factprop} and Fact~\ref{SS:rif} 
once we realize that the numerator of the rational inner function given by \eqref{E:altform}
{\em cannot} vanish at $0$ if $\nu(\que)\in {\sf supp}(\que)$.
We shall not write down the (essentially trivial, in view of Proposition \ref{P:factprop}) proof of this
corollary.
\smallskip

The proof of Proposition \ref{P:factprop} depends on a few lemmas. The first of these
states a simple factorization property associated to 
$Q$ and $z^{\nu(Q)}\widetilde{Q}(\frac{1}{z}).$

\begin{lemma}\label{L:conirr}
Let $Q$ be a nonconstant polynomial such that $Q(0)\not =0.$
Then $z^{\nu(Q)}\widetilde{Q}(\frac{1}{z})$ is irreducible 
in $\C[z_1,\dots,z_n]$ if and only if $Q$ is irreducible in $\C[z_1,\dots,z_n].$
\end{lemma}

This is an entirely elementary result which depends on the simple calculation that if 
$Q=Q_1Q_2,$ $Q_1$ and $Q_2$ being nonconstant polynomials, then 
\begin{equation}
z^{\nu(Q)}\widetilde{Q}\big(1/z\big)=
[z^{\nu(Q_1)}\widetilde{Q}_1\big(1/z\big)]
[z^{\nu(Q_2)}\widetilde{Q}_2\big(1/z\big)]
\label{E:red1}
\end{equation}
for all $z\in\Cn\setminus(\cup_{j=1}^{n}\{z\in\Cn:z_j=0\}),$ and hence on all of $\Cn.$ 
We shall not dwell any further on this simple matter.
\smallskip

The next two results follow easily if we make use of the work
of Agler--McCarthy--Stankus \cite{agmacsta:ftp06}. 
To use their results, we need to give definitions of several terms from \cite{agmacsta:ftp06}.

\begin{definition}[see Section~2 of \cite{agmacsta:ftp06}]
An {\em algebraic set in $\Cn$} is the intersection of the zero sets of finitely many polynomials 
in $\C[z_1,\dots,z_n].$
\begin{itemize}
\item[$(1)$] Given an algebraic set $A\subset\Cn,$ ${\sf Hol}(A)$ will denote the
algebra of all functions $f:A\rightarrow\C$ such that for each point $z\in A,$ there is an open subset
$U^{z}$ of $\Cn$ containing $z,$ and a holomorphic function $\varphi_{z}$ on $U^{z}$ such that 
$\varphi_{z}|_{U^{z}\cap A}\equiv f|_{U^{z}\cap A}.$
\item[$(2)$] Let $X\subseteq\Cn,$ and let $A$ be an algebraic set in $\Cn.$ We say $X$ is 
{\em determining for $A$} if for every $f\in{\sf Hol}(A)$ with $f|_{X\cap A}\equiv 0$ we have that 
$f\equiv 0.$ 
\item[$(3)$] We say that an algebraic set $A$ is {\em toral} if $\Tn$ is determining for $A$, and that $A$
is {\em atoral} if $\Tn$ is not determining for any of the irreducible components of $A.$
If $Q\in\C[z_1,\dots,z_n],$ we say that $Q$ is {\em toral} (resp.\;{\em atoral}) if $Z(Q)$ is toral (resp.\;atoral).
\end{itemize}
\end{definition}

The empty set is both toral and atoral. Consequently, the nonzero constant polynomials 
are both toral and atoral. We are now in a position to state the next lemma required.

\begin{result}(paraphrasing Corollary $3.2$ of \cite{agmacsta:ftp06})\label{R:factato}
Let $Q\in\C[z_1,\dots,z_n]$ be a nonzero polynomial. There exists a factorization 
\begin{equation}
Q=pq,\label{E:factato}
\end{equation}
where $p$ is toral and $q$ is atoral, and $p$ and $q$ are determined uniquely up to constants.
\end{result}

When $Q$ is a nonconstant polynomial, we call the polynomial $q$
given by the factorization \eqref{E:factato} {\em the atoral factor of $Q$}
(with the understanding that $q$ is determined upto a constant).

\begin{lemma}\label{L:nu}
Let $f$ be a nonconstant rational inner function, and write
\begin{equation}
f(z)=\frac{Az^{\beta}\widetilde{Q}(\frac{1}{z})}{Q(z)},\label{E:RI}
\end{equation}
where $Q$ is nonconstant, and  $A,$ $\beta$ and $Q$ have exactly the meanings and properties
stated under the heading ``Fact~\ref{SS:rif}'' above. 
Then $f$ has a zero in $\D^n.$ In particular, the numerator of \eqref{E:RI} has a zero in $\D^n.$
\end{lemma}

\begin{proof}
Writing $f$ as
\[
 f(z) = A z^{\beta-\nu(Q)}\,\frac{z^{\nu(Q)}\widetilde{Q}(\frac{1}{z})}{Q(z)},
\]
we see that it suffices to assume without loss of generality that $\beta-\nu(Q)=0.$
Let $\mathbb{E}:=\{z\in\mathbb{C}: |z|>1\}.$ Assume $f$ does not have any zeros in $\D^n.$ Then 
${z^{\nu(Q)}\widetilde{Q}(\frac{1}{z})}$ does not have any zeros in $\D^n.$ Hence, 
$Q$ cannot have any zeros in $\mathbb{E}^n.$ So $Q$ is a nonconstant polynomial such that
$Z(Q)\cap\D^n=\emptyset$ and $Z(Q)\cap\mathbb{E}^n=\emptyset.$ It follows from
\cite[Theorem 3.5]{agmacsta:ftp06} that $Q$ is toral. Hence, by \cite[Proposition 3.4]{agmacsta:ftp06},
there exists a $c\in \C$ such that $Q(z)=c{z^{\nu(Q)}\widetilde{Q}(\frac{1}{z})}$.
This implies that $f$ is constant, which is a contradiction. 
\end{proof}

We now have all the tools to present the proof of the Proposition \ref{P:factprop}.

\begin{proof}[{\bf Proof of the Proposition \ref{P:factprop}}]
In this proof, all ring-theoretic assertions made {\em without any further qualification will be
for the ring $\C[z_1,\dots,z_n].$}
\smallskip

We factor $Q = pq$ according to Result~\ref{R:factato}. Then, from the discussion leading to the
equation \eqref{E:red1}, and from the argument towards the end of the previous proof\,---\,now
applied to the toral polynomial $p$\,---\,we get\vspace{-1.5mm} 
\begin{equation}\label{E:altform2}
 f(z) = Cz^{\nu(q)}\,\widetilde{q}\,\big(1/z\big)\big/q(z),
\end{equation}
where $C$ is some unimodular constant. Moreover, it follows from the argument in the first paragraph
of the proof of Theorem 4.1 in \cite{agmacsta:ftp06} that $q$ does not have any nonconstant irreducible
factor $r$ satisfying the identity $r(z)=c{z^{\nu(r)}\widetilde{r}(\frac{1}{z})}$ ($c$ being a
non-zero constant).
\smallskip

Now, from the third paragraph of the proof of Theorem 4.1 in \cite{agmacsta:ftp06}, we conclude that the
numerator and the denominator of the right-hand side in \eqref{E:altform2} are relatively prime to each other. 
Thus, we have $(a)$, with $\que$ being some choice of the atoral factor (which is determined uniquely up to
a constant) of $Q$.
\smallskip
 
Suppose $\que$ is reducible in $\C[z_1,\dots, z_n].$ Then there exist $q_1, q_2\in
\C[z_1,\dots, z_n]$ which are nonunits such that $\que=q_1q_2.$ As $Z(\que)\cap\D^n=\emptyset,$
we have $Z(q_i)\cap\D^n=\emptyset,\,\,i=1,2.$ By an elementary
calculation that we had alluded to earlier, which leads to \eqref{E:red1}, we have:
\[
z^{\nu(\que)}\,\widetilde{\que}\big(1/z\big)\big/\que(z)
=[z^{\nu(q_1)}\,\widetilde{q}_1\big(1/z\big)\big/q_1(z)]
 [z^{\nu(q_2)}\,\widetilde{q}_2\big(1/z\big)\big/q_2(z)].
\]
The properties of $\que$ enable us to apply Lemma \ref{L:nu} to each of the factors
of the right-hand side in the above equation and to infer that they are nonunits 
of $\mathcal{O}(\D^n).$ Clearly they are rational inner.
This gives us one of the implications in $(b)$.
\smallskip

Now assume there exist $f_1,f_2,$ rational inner and nonunits in $\hol(\D^n)$
such that $f\equiv f_1f_2.$ Owing to Fact~\ref{SS:rif}, we can write
\[
f_i(z) = [A_i z^{\beta_i-\nu(Q_i)}\,][{z^{\nu(Q_i)}\widetilde{Q}_i({1}/{z})}/{Q_i(z)}], \ \ \ 
i=1, 2,
\]
where $A_i,$ $\beta_i$ and $Q_i$ are as described in Fact~\ref{SS:rif}.
Put $P_i(z):= A_i z^{\beta_i}\widetilde{Q}_i(\frac{1}{z}),\,\,i = 1,2.$
In view of $(a)$, we can assume without loss of generality that $\hcf(P_i,Q_i)=1,\,\,i=1,2$
(note that we do {\bf not} require $Q_i$ to be nonconstant to assert this). 
This assumption will be in effect for the remainder of this proof.
\smallskip

Set $\mathcal{P}(z):=z^{\nu(\que)}\widetilde{\que}\left(\frac{1}{z}\right)$. 
Appealing to Lemma~\ref{L:nu} if $Q_i$ is nonconstant, else to the fact that $f_i$ is nonconstant,
we deduce that $P_1$ and $P_2$ have zeros in $\D^n$. We have
\begin{equation}
C\,\frac{\mathcal{P}}{\que}=\frac{P_1P_2}{Q_1Q_2}=\frac{p_1p_2}{q_1q_2}, \label{maineqid}
\end{equation}
where $p_1$ and $q_2$ are obtained by cancelling any common factors that $P_1$ and $Q_2$ might
have; and defining the pair $p_2$ and $q_1$ analogously.  
Any such nonconstant common factor cannot have zeros in $\D^n.$ Hence,
$p_1$ and $p_2$ must have zeros in $\D^n$ and are nonunits in $\C[z_1,\dots,z_n].$  
Now \eqref{maineqid} gives us
\begin{equation*}
C\mathcal{P}q_1q_2 = \que p_1p_2.
\end{equation*}
Hence $p_1p_2|\mathcal{P}q_1q_2.$ As $\hcf(p_1p_2, q_1q_2)=1,$ we have $p_1p_2|\mathcal{P},$
whence $\mathcal{P}$ is reducible. Hence from Lemma \ref{L:conirr}, $\que$ is reducible. This establishes $(b)$.
\end{proof}
\medskip

\section{The proof of Theorem~\ref{T:mainThm}}\label{S:Thm}
Before presenting the proof of Theorem \ref{T:mainThm}, we shall state a few
results that are essential to our proof.

\subsection{A few essential results}\label{SS:Bbt}
The first result that we need is not stated explicitly as a theorem by Knese, but is the content
of the discussion in \cite[Section 4]{Knese:ftp015}.

\begin{result}[Knese, \cite{Knese:ftp015}]\label{Re:Kn}
Let $X_1,X_2,X_3$ be three distinct points in $\D^n,\ n\geq 2,$ and let $w_1,w_2,w_3\in\D.$ Suppose
there exists $\Phi\in H^{\infty}(\D^n)$ such that $\sup_{\D^n}|\Phi|\leq 1$ and
$\Phi(X_j)=w_j,\ j=1,2,3.$ Then there exists $\widetilde{\Phi}$ in the Schur--Agler class such
that $\widetilde{\Phi}(X_j)=w_j,\ j=1,2,3.$
\end{result}

The next result has actually been proved in \cite{agmac:ftp99} for the Schur class on the 
bidisc. However, the result \emph{as stated below} can be seen to follow, along the same
lines as the proof of \cite[Corollary~2.13]{agmac:ftp99}, if one uses the representation theorem
for the Schur--Agler class on the polydisc \cite{Agler:ftp90}. (Since the latter theorem has quite a long
statement, we shall not spell it out.) Also see the remark on p.~193 of \cite{agmac:ftp99}.

\begin{result}[Agler--McCarthy, \cite{agmac:ftp99}]\label{Re:AgMc}
Suppose the problem $(*)$ has a solution in the Schur--Agler class. Then there exists a 
rational inner function that solves $(*).$
\end{result}

In $(*),$ when $n=1$ we note that $X_j\in\D.$ In this case it is known 
that a solution for $(*)$ exists if and only if the matrix
\begin{equation}
{\begin{bmatrix}
\dfrac{1-w_j\bar{w}_k}{1-X_j\overline{X}_{k}}
\end{bmatrix}}_{j,k=1}^{N}\label{E:NP1}
\end{equation}
is positive semi-definite.
The last result needed\,---\,which concerns the one-dimensional Pick--Nevanlinna problem\,---\,is
about the existence of interpolants of a certain special form when the matrix in \eqref{E:NP1}
is positive semi-definite.

\begin{result}[Theorem 6.15, \cite{agmac:ftp02}]\label{Re:Blaint}
Suppose the matrix in \eqref{E:NP1} is positive semi-definite and has rank $r,\,r\leq N.$ Then there exists
a finite Blaschke product of degree $r$ that solves the problem $(*).$
\end{result}
\smallskip

\subsection{The proof of Theorem \ref{T:mainThm} and consequences}
We are now in a position to present the proof of our main theorem.
In the remainder of this section, we will use expressions of the form
``a function that interpolates the data $(X_1,\dots,X_N;w_1,\dots,w_N)$''
to signify the existence of a function, in the stated class, that maps
the data in the manner described by $(*).$
The following lemma is another key tool in our proof of Theorem~\ref{T:mainThm}.

\begin{lemma}\label{L:twointer}
Let $(X_1,w_1),(X_2,w_2)\in\D^n\times\D.$ There exists a holomorphic map in $\hol(\D^n,\D)$
interpolating the data $(X_1,X_2;w_1,w_2)$ if and only if
\begin{equation}
C_{\D^n}(X_1,X_2)\geq C_{\D}(w_1,w_2),\label{E:cath2}
\end{equation}
where $C_{\D^n}$ and $C_{\D}$ denote the Carath{\'e}odory distance on $\D^n$ and $\D$ respectively.
\end{lemma}

We give only a sketch of the proof of the lemma above.
The ``only if'' part follows from the well-known distance decreasing property of the
Carath{\'e}odory distance under holomorphic maps. We refer the reader to \cite[Chapter 2]{Jar:ftp93} for the
definition and the basic properties of the Carath{\'e}odory distance used in this and the next paragraph.
\smallskip

It is a standard calculation that
\[
 C_{\D^n}(X_1,X_2) = \max\{C_{\D}(X_{1,\,j},X_{2,\,j}) : 1\leq j\leq n\}.
\]
So if \eqref{E:cath2} holds true, 
then there exists an $l$, $1\leq l\leq n,$ such that
\[
 C_{\D}(X_{1,\,l}, X_{2,\,l}) \geq C_{\D}(w_1,w_2).
\]
Now, $C_{\D}(\zeta, \eta)$ is just the Poincar{\'e} distance between the points $\zeta, \eta\in \D$.
Thus, by the last inequality we can explicitly construct a function $f$\,---\,which is the conjugation of an appropriate
scaling by an automorphism of $\D$\,---\,such that $f(X_{1,\,l})=w_1$ and $f(X_{2,\,l})=w_2.$
Then $f\circ\pi_l$ interpolates as desired, where $\pi_l$ is the projection onto the $l$-th coordinate.
\smallskip
 
We now have all the tools to present the proof of Theorem~\ref{T:mainThm}.

\begin{proof}[{\bf Proof of Theorem~\ref{T:mainThm}}] 
The implication (\rom{2})\,$\Rightarrow$\,(\rom{1})
is obvious. Now assume (\rom{1}). Then by Result \ref{Re:Kn}
there exists a $\widetilde{\Phi}$ in the Schur--Agler class such that
$\widetilde{\Phi}(X_j)=w_j,\,\,j=1,2,3.$ Hence, by Result \ref{Re:AgMc}
there exists a rational inner function $F$ that interpolates the given data.
This establishes the implication (\rom{1})\,$\Rightarrow$\,(\rom{2}).
\smallskip

The remainder of the proof deals with the equivalence of (\rom{2}) and (\rom{3}).
To this end, let $F\in\hol(\D^n)$ be a rational inner function
that interpolates the data $(X_1,X_2,X_3;w_1,w_2,w_3).$ Then the interpolant $F$ exists if and only if
$\widetilde{F}:=\psi_{w_3}\circ F\circ{\Psi_{X_3}^{-1}},$ which is a rational inner function on $\D^n,$
interpolates the data $(X_1',X_2',0;w_1',w_2',0),$ where $X_1',X_2',w_1'$ and $w_2'$
are as stated in the theorem. Here $\psi_{w_3}$ and $\Psi_{X_3}$ are as introduced in Section~\ref{S:intro}.
\smallskip

\noindent{\bf Claim.} {\em The interpolant $\widetilde{F}$ exists if and only if there exist $H,G,$
both rational inner functions on $\D^n,$ with $H$ having the form described in Theorem~\ref{T:mainThm}, 
such that $G$ interpolates $(X_1',X_2';{w_1'}/{H(X_1')},{w_2'}/{H(X_2')}),$ and such that 
${w_j'}/{H(X_j')}\in\overline{\D}$ for $j=1,2.$}

\noindent The ``if'' part of the above claim is easy to prove. Assume that $G,H$ exist as in the claim. Then take
$\widetilde{F}=GH,$ which has all the desired properties.
\smallskip

To see the ``only if'' part we consider two cases.
In what follows, the adjectives {\em irreducible} and {\em reducible},
applied to $\widetilde{F},$ are as defined prior to Corollary \ref{C:irdefpol}.
\smallskip

\noindent{\bf Case 1.} {\em The interpolant $\widetilde{F}$ is irreducible.}
 
\noindent In this case we take $H=\widetilde{F}$ and $G\equiv 1.$
Note that both are rational inner functions. That $H$ has the form described in Theorem~\ref{T:mainThm}
follows from Corollary \ref{C:irdefpol}.
\smallskip

\noindent{\bf Case 2.} {\em $\widetilde{F}$ is reducible.}
 
\noindent Since $\widetilde{F}$ is reducible, and $\widetilde{F}(0)=0,$ there exist an irreducible rational inner function
$H$ such that $H(0)=0,$ and a rational inner function $G$ such that $\widetilde{F}=GH.$ In view of
Corollary~\ref{C:irdefpol}, $G$ and $H$ have the properties claimed. 
\smallskip

This establishes our Claim.
\smallskip

Let us look closely at the situation in Case~2. Since $X_j'\in\D^n$ for $j=1,2,$
we have $|{w_j'}/{H(X_j')}|=|G(X_j')|<1,\ j=1,2.$
We have used here the fact that $G$ is nonconstant. We have from Lemma~\ref{L:twointer} that the 
existence of $G$ and $H$ as in our Claim leads to
\begin{equation}
C_{\D^n}(X_1',X_2')\geq C_{\D}\left(\frac{w_1'}{H(X_1')},\frac{w_2'}{H(X_2')}\right). \label{E:card}
\end{equation}
As $C_{\D^n}(X_1',X_2')=\max\{C_{\D}(X_{1,\,j}',X_{2,\,j}'):1\leq j\leq n\},$
the inequality \eqref{E:card} is equivalent to 
\begin{equation}
C_{\D}(X_{1,\,l}',X_{2,\,l}')\geq C_{\D}\left(\frac{w_1'}{H(X_1')},\frac{w_2'}{H(X_2')}\right)
\ \ \ \text{for some $l,1\leq l\leq n.$}\label{E:cardco}
\end{equation}
Writing the expression for $C_{\D},$ a simple matricial trick (see \cite[page 7]{garnett:ftp07})
shows that the inequality \eqref{E:cardco} is equivalent to
\begin{equation}
{\begin{bmatrix}
\dfrac{1-({w_j'}/{H(X_j')})(\overline{{w_k'}/{H(X_k')}})}{1-X_{j,\,l}'\overline{X'}_{k,\,l}}
\end{bmatrix}}_{j,k=1}^{2}
\geq 0.\label{E:mainm}
\end{equation}

The interpolation criterion in (\rom{3}) is stated
in terms of a quadratic form because \eqref{E:card} does not make sense in Case~1.
In Case~1 the existence of the interpolant $\widetilde{F}$ implies that the interpolant $G$ is 
the constant $1,$ whence $w_j'/H(X_j')=1,\ j=1,2.$ Trivially, the matrix in \eqref{E:mainm}
is positive semi-definite. This establishes the implication (\rom{2})\,$\Rightarrow$\,(\rom{3}). 
\smallskip

Let us denote the matrix in \eqref{E:mainm} by $M_l.$ In view of the chain of equivalences discussed above,
the implication (\rom{3})\,$\Rightarrow$\,(\rom{2}) will follow if we can produce a
rational inner function $G$ with the properties stated in our Claim.
So we assume that $M_l$ is positive semi-definite (which tacitly assumes the existence of the function $H$
with the properties stated above). Using Result \ref{Re:Blaint}, we get  
a finite Blaschke product $B$ (which includes the case when $B$ is a unimodular constant)
that interpolates the data $(X_{1,\,l}',X_{2,\,l}';w_1'/H(X_1'),w_2'/H(X_2')).$ Take $G=B\circ \pi_l,$ where $\pi_l$ 
denotes the projection onto the $l$-th coordinate. This $G$ satisfies all the properties as in the above Claim.
\smallskip

Suppose, now, that the condition in (\rom{3}) holds true. Then it is easy to see that the matrix in 
\eqref{E:mainm} is the zero matrix if and only if $w_1'/H(X_1')=w_2'/H(X_2')=c\in\partial{\D}.$ It follows from
the discussion in the previous paragraph that
$\widetilde{F}=\psi_{w_3}\circ F\circ\Psi^{-1}_{X_3}=cH,$ and $(a)$ follows from this. When the rank $r\geq 1$,
we refer to the full force of Result \ref{Re:Blaint}: this gives the degree of the Blaschke product $B$
mentioned in the previous paragraph. 
Arguing as before, $\widetilde{F}=(B\circ\pi_l)H,$ and we are done.
\end{proof}

\subsection{Implications of the proof of Theorem~\ref{T:mainThm}}\label{SS:impli}
We begin by elaborating upon the point made in Remark \ref{Rm:m}. Suppose, by some method,
one has a rational inner interpolant $F_1$ for the data $(X_1,X_2,X_3;w_1,w_2,w_3)$ having a high
degree. By ``degree'' here, we refer to $|\nu(\que_1)|$, where $\que_1$ is the polynomial
associated to $F_1$ with the properties given by part~$(a)$ of Proposition~\ref{P:factprop}. {\em If
$F_1$ is a reducible inner function}, as defined prior to Corollary~\ref{C:irdefpol}, then let
$\widetilde{F_1}:=\psi_{w_3}\circ F_1\circ{\Psi_{X_3}^{-1}}$ and write                                                        
\[
 \mathfrak{S}\,:=\,\text{the set of irreducible rational inner factors $H$ of $\widetilde{F_1}$ such that $H(0) = 0$.}
\]

At this stage, one might be tempted to say that {\em generically $F_1$ is irreducible} (recall that $n\geq 2$),
whence the circumstances under which the algorithm below can be used occur rarely. However, this preliminary
insight does not capture the realities of the interpolation problem $(*)$ in their entirety. For instance, we have that:
\begin{itemize}
 \item[$(\bullet)$] The subset of $(\D^n)^N\times\D^N$
 of those $N$-point Pick--Nevanlinna data that admit an irreducible rational inner interpolant contains a
 non-empty relatively-open subset $\OM$ such that each of the
 data-points in $\OM$ {\em also} admits a reducible rational interpolant.
\end{itemize}
We shall elaborate upon this\,---\,as well as on a further point about the utility of the following algorithm\,---\,in
Remark~\ref{Rm:Gen} below. At present, we resume the discussion begun above.
\smallskip 
     
For $H\in \mathfrak{S}$, let $\que_H$ be the polynomial associated to $H$ with the properties given by
part~$(a)$ of Proposition~\ref{P:factprop}. Let $H^*$ be such that
\[
 |\nu(\que_{H^*})|\,=\,\min\{|\nu(\que_H)| : H\in \mathfrak{S}\}.
\]
Let $M_l(H^*)$ denote the matrices of the form given by \eqref{E:mainm} with $H = H^*$ and
let
\[
\mathscr{S}:=\{1\leq l\leq n: M_l(H^*)\,\,\text{is positive semi-definite}\}.
\] 
 By our proof above, we have that $\mathscr{S}\neq \emptyset$. Let $m = \min\{\text{rank}(M_l(H^*)) :
 l\in \mathscr{S}\}$ and let $l_{0}$ be such that  
$m=\text{rank}(M_{l_0}(H^*)).$ There will exist a Blaschke product $B$ of degree $m$ that interpolates
the data $(X_{1,\,{l_0}}', X_{2,\,{l_0}}'; w_1'/H^*(X_1'),$ $w_2'/H^*(X_2'))$\,---\,this is a
consequence of Result \ref{Re:Blaint}. Write:
\[
 F_2:={\psi}^{-1}_{w_3}\circ((B\circ \pi_{l_0})H^*)\circ\Psi_{X_3}.
\]
In the event that $|\nu(\que_1)| > |\nu(\que_{H^*})| + m$, $F_2$ is an interpolant of
strictly lower ``degree''.

\begin{remark}\label{Rm:Gen}
We begin with a brief justification of $(\bullet)$. Define
\[
 \Delta\,:=\,\{\!(X_1,\dots, X_N)\in (\D^n)^N : X_{1,\,l},\dots, X_{N,\,l} \ \text{are distinct for each
 $l=1,\dots, n$}\}.
\]
Now consider the set
\[   
 \mathcal{S}\,:=\,\bigg\{\!(X_1,\dots, X_N; w_1,\dots, w_N)\in \Delta\times(\D)^N: 
 			\left[\frac{1-\overline{w}_k w_j}{1-\overline{X}_{k,\,l}X_{j,\,l}}\right]_{j,\,k = 1}^N\!\!\geq 0 \ \text{for
 			some $l= 1,\dots, n$}\bigg\}.
\]
It is easy to see that $\mathcal{S}$ has non-empty interior. It follows from the classical Pick--Nevanlinna
theory that the generic $(X_1,\dots, X_N; w_1,\dots, w_N)$ in $\mathcal{S}$ admits a {\bf reducible} rational
inner interpolant. This is because membership in $\mathcal{S}$ implies the existence of an interpolant that is the composition of a
Blaschke product with the projection onto the $l$-th coordinate for some $l$. Indeed, for each point in $\mathcal{S}^\circ$, at least
one of the associated matrices\,---\,indexed by $l = 1,\dots, n$\,---\,appearing in the definition of $\mathcal{S}$ is of
rank $N$. It follows from Result~\ref{Re:Blaint} that each such point admits an interpolant for which the Blaschke
product referred to above is of degree $N$. From this, $(\bullet)$ follows. This has a bearing on the point\,---\,raised
above\,---\, that
``generically $F_1$ is irreducible,'' where $F_1$ is as in the beginning of the present subsection. The latter is a principle
that relies on the fact that the class of rational inner functions of degree\,$\leq d$
(where ``degree'' is as explained at
the beginning of this subsection) is parametrized\,---\,thanks to Fact~\ref{SS:rif}\,---\,by a real-analytic (hence
stratified) set whose generic stratum is the manifold 
$\{\zt\in \C: |\zt| = 1\}\times V_d$, where
\[
  V_d\,:=\,\{P\in \C[z_1,\dots,z_n] : {\rm deg}(P)\leq d \ \text{and} \ Z(P)\cap \D^n = \emptyset\},
\]
and that irreducible polynomials are generic in $V_d$ (since $n\geq 2$). {\em However}, this is, largely, not
germane to the problem $(*)$, even for $N = 3$, because:
\begin{itemize}
 \item It is unclear that the set
 of Pick--Nevanlinna data admitting {\em irreducible} rational inner interpolants is generic in the subset of
 $(\D^n)^3\times\D^3$ of all $3$-point Pick--Nevanlinna data that admit an interpolant of the Schur class
 on $\D^n$.
 \item As we have just discussed, there is a ``large'' set\,---\,contained within the set
 of all $N$-point Pick--Nevanlinna data that admit an interpolant of the Schur class
 on $\D^n$\,---\,of Pick--Nevanlinna data that are interpolated by reducible rational inner functions.
\end{itemize}
In short, it will not be uncommon, given some data $(X_1,X_2,X_3;w_1,w_2,w_3)$ that admits an
interpolant of the Schur class on $\D^n$, to obtain, by numerical methods or otherwise, a rational
inner interpolant that is reducible. The algorithm discussed above will then apply to such an interpolant.    
\end{remark}

We had hinted earlier that the proof of (\rom{2})\,$\Leftrightarrow$\,(\rom{3})
contains ideas for extending our result\,---\,at least when rational interpolants are sought\,---\,to
$N$ points, $N\geq 4.$ We shall end with some remarks on this issue. 
For this purpose, we need the following result about the positivity of
certain quadratic forms. The 
proof of the result is found in the body of the proof of Theorem 2.2 in Garnett \cite{garnett:ftp07}.
\begin{result}\label{R:quadposi}
Let $\{(a_j,b_j)\in\D\times\D: 1\leq j\leq n\},$ where $a_j$'s are distinct.
Let $a_j'=\psi_{a_n}(a_j)$ and $b_j'=\psi_{b_n}(b_j),\,\,1\leq j\leq n.$
Consider the quadratic form:
\begin{equation*}
Q_n(t_1,t_2,\dots,t_n):=\sum_{j,k=1}^{n}\frac{1-b_j\bar{b}_k}{1-a_j\bar{a}_k}t_j\bar{t}_k.
\end{equation*}
Let $Q_n'$ be the quadratic form obtained from $Q_n$ by replacing $a_j$ with $a_j'$ and $b_j$ with $b_j'.$ Then
\begin{equation*}
Q_n\geq 0 \iff Q_n'\geq 0.
\end{equation*} 
Moreover if we take $a_n=0=b_n$ in $Q_n$, and consider the quadratic form:
\begin{equation*}
\widetilde{Q}_{n-1}(s_1,s_2,\dots,s_{n-1})=
\sum_{j,k=1}^{n-1}\frac{1-({b_j}/{a_j})(\overline{{b_k}/{a_k}})}{1-a_j\bar{a}_k}s_j\bar{s}_k,
\end{equation*}
then
\begin{equation*}
Q_n\geq 0 \iff \widetilde{Q}_{n-1}\geq 0 \\\ \text{(taking $a_n=0=b_n$ in $Q_n$)}.
\end{equation*}
\end{result}

\begin{remark}\label{Rm:E}
Using Result \ref{R:quadposi} it is possible to replace the positive semi-definiteness of the matrix in \eqref{E:mainm}
by the positive semi-definiteness of a certain $3\times 3$ matrix where, in the denominator of each entry
$X_{j,\,l}, \ j=1, 2, 3,$ appear.
The proof of (\rom{2})\,$\Leftrightarrow$\,(\rom{3})
together with the above discussion that ``inflates''
condition \eqref{E:mainm}
into a condition on a $3\times 3$ matrix, whose form is suggestive, points to
a generalization for $N$ points, $N\geq4,$
involving the positivity of an $N\times N$ matrix. Our proof would proceed by deflating the   
$N$-point data set to an equivalent $(N-1)$-point data set, which sets up an inductive scheme as
in the Schur algorithm. But this results in a condition that is perhaps too unwieldy to be useful.
We will not present this generalization here as the associated technicalities only obscure
the main idea underlying this work.
\end{remark}


\medskip


\begin{thebibliography}{1}
\bibitem{ag:ftp88} Jim Agler, Interpolation, unpublished manuscript, 1988.

\bibitem{Agler:ftp90}
Jim Agler, {\em On the representation of certain holomorphic functions defined on a polydisc}, Topics in 
Operator Theory: Ernst D. Hellinger Memorial Volume, 47-66, Oper. Theory Adv. Appl., {\bf 48}, Birkh{\"a}user,
Basel, 1990.

\bibitem{agmac:ftp99} Jim Agler, John E. McCarthy, {\em Nevanlinna--Pick interpolation on the bidisk}, J. Reine
Angew. Math. {\bf 506} (1999), 191-204.

\bibitem{agmac:ftp02} Jim Agler, John E. McCarthy, Pick Interpolation and Hilbert Function Spaces,
American Mathematical Society, 2002.

\bibitem{agmacsta:ftp06} Jim Agler, John E. McCarthy, Mark Stankus, {\em Toral algebraic sets and
function theory on polydisks}, J. Geom. Anal. {\bf 16} (2006), no. 4, 551-562.

\bibitem{ando:ftp63} T. And\^{o}, {\em On a pair of commutative contractions}, Acta Sci. Math.
(Szeged) {\bf 24} (1963), 88-90.

\bibitem{ball:ftp98} Joseph A. Ball, Tavan T. Trent, {\em  Unitary colligations, reproducing kernel Hilbert spaces,
and Nevanlinna--Pick interpolation in several variables}, J. Funct. Anal. {\bf 157} (1998), no. 1, 1-61.

\bibitem{ball:ftp99} Joseph A. Ball, W. S. Li, D. Timotin, Tavan T. Trent,
{\em  A commutant lifting theorem on the polydisc}, Indiana Univ. Math. J. {\bf 48} (1999), no. 2, 653-675. 

\bibitem{cw:ftp94} B.J. Cole, J. Wermer, {\em Pick interpolation, von Neumann inequalities, and hyperconvex sets},
Complex Potential Theory, Kluwer Acad. Publ., Dordrecht, 1994, 89-129.

\bibitem{chand:ftp015} Vikramjeet Singh Chandel, {\em The three-point Pick--Nevanlinna interpolation problem on
the polydisc}, preprint (2015), arXiv reference: \texttt{arXiv:1505.03763v1}.

\bibitem{garnett:ftp07} John B. Garnett, Bounded Analytic Functions, Springer Verlag, 2007.

\bibitem{ham:ftp013} R. Hamilton, {\em Pick interpolation in several variables}, Proc. Amer. Math. Soc.
{\bf 141} (2013), no. 6, 2097-2103.

\bibitem{Jar:ftp93} Marek Jarnicki, Peter Pflug, Invariant Distances and Metrics in Complex Analysis,
de Gruyter Expositions in Mathematics, 9, Walter de Gruyter \& Co., Berlin, 1993.

\bibitem{Knese:ftp015} Greg Knese, {\em The von Neumann inequality for $3\times 3$ matrices}, to appear in
Bull. London Math. Soc.

\bibitem{kos:ftp015} Lukasz Kosi{\'n}ski, {\em Three-point Nevanlinna Pick problem in the polydisc},
 Proc. Lond. Math. Soc. (3) {\bf 111} (2015), no. 4, 887-910.

\bibitem{rudin:ftp69} Walter Rudin, Function Theory in Polydiscs, W.A. Benjamin, 1969.

\bibitem{sarason:ftp67} D. Sarason, {\em Generalized interpolation in $H^{\infty}$}, Trans. Amer. Math. Soc.
{\bf 127} (1967), 179-203.


\end{thebibliography}
\end{document}